\documentclass[11pt]{article}
\usepackage{amsmath,amssymb}

\newtheorem{propo}{{\bf Proposition}}[section]

\newtheorem{lemma}[propo]{{\bf Lemma}} \newtheorem{theor}[propo]{{\bf
Theorem}} \newtheorem{ex}{{\sc Example}}[section]

\newenvironment{proof}{{\bf Proof.}}{$\Box$}
\def\R{{\mathbb R}}

\begin{document}

\vspace*{1.0in}

\begin{center}ON MINIMAL NON-ELEMENTARY LIE ALGEBRAS  
\end{center}
\bigskip

\begin{center} DAVID A. TOWERS 
\end{center}
\bigskip

\begin{center} Department of Mathematics and Statistics

Lancaster University

Lancaster LA1 4YF

England

d.towers@lancaster.ac.uk 
\end{center}
\bigskip

\begin{abstract} The class of minimal non-elementary Lie algebras over a field $F$ are studied. These are classified when $F$ is algebraically closed and of characteristic different from $2,3$. The solvable algebras in this class are also characterised over any perfect field.
\par 
\noindent {\em Mathematics Subject Classification 2000}: 17B05, 17B20, 17B30, 17B50.
\par
\noindent {\em Key Words and Phrases}: Lie algebras, solvable, Frattini ideal, elementary algebras. 
\end{abstract}

\section{Introduction}
Groups and Lie algebras that just fail to have a particular property have been studied extensively in the hope of gaining some insight into just what makes the group or algebra have that property.  They are also useful in constructing induction proofs. In group theory such groups are sometimes called {\em critical}.
\par

A Lie algebra $L$ is called {\em elementary} if the Frattini ideal, $\phi(S)$, of each of its subalgebras $S$, is trivial. It is called {\em minimal non-elementary} if each of its proper subalgebras is elementary but it is not elementary itself. These algebras were first studied by Towers in \cite{elem1}, and recently Stagg and Stitzinger (\cite{stit}) have classified such algebras when $L^2$ is nilpotent and the ground field is algebraically closed. In section 2 we extend this result first to any solvable Lie algebra over an algebraically closed field, and then show that there are no such non-solvable Lie algebras provided that the underlying field also has characteristic different from $2,3$. The final result characterises all minimal non-elementary solvable Lie algebras over a perfect field. An example is given to show that there are minimal non-elementary Lie algebras over the real field that are not of the type described in the result of Stagg and Stitzinger. 
\par

Throughout, $L$ will denote a finite-dimensional Lie algebra over a field $F$. Algebra direct sums will be denoted by $\oplus$, whereas direct sums of the vector space structure alone will be denoted by $\dot{+}$.

\section{The results}
We say that $L$ is an {\em $E$-algebra} if $\phi(S) \subseteq \phi(L)$ for every subalgebra $S$ of $L$.

\begin{lemma} Let $L$ be a minimal non-elementary Lie algebra over a perfect field $F$. Then $L$ is solvable if and only if $L^2$ is nilpotent.
\end{lemma}
\begin{proof}  Clearly $L$ must be an $E$-algebra and so the result follows from \cite[Corollary 2.2]{elem2}.
\end{proof}
\bigskip

The above lemma easily yields the following generalisation of  \cite[Theorem]{stit}. 

\begin{theor}\label{t:solvalgclosed} Let $L$ be a solvable Lie algebra over an algebraically closed field $F$. Then $L$ is minimal non-elementary if and only if $L$ has a basis $x, y, z$ with multiplication $[x,y] = \alpha y + z$, $[x,z] = \alpha z$ and $[y,z] = 0$, where $\alpha \in F$.
\end{theor}

The algebra $L$ is called an {\em $A$-algebra} if all of its nilpotent subalgebras are abelian. Then the following result is a special case of \cite[Proposition 2]{ps}.

\begin{theor}\label{t:ps} (Premet and Semenov) Let $L$ be a non-solvable $A$-algebra over an algebraically closed field of characteristic different from $2,3$. Then  $L \cong (S_1 \oplus \ldots \oplus S_n) \dot{+} R$, where $n \geq 1$, $S_i \cong sl_2(F)$ for $1 \leq i \leq n$ and $R$ is the (solvable) radical of $L$.
\end{theor}

\begin{theor}\label{t:algclosed} Let $L$ be any Lie algebra over an algebraically closed field $F$ of characteristic different from $2,3$. Then $L$ is minimal non-elementary if and only if $L$ has a basis $x, y, z$ with multiplication $[x,y] = \alpha y + z$, $[x,z] = \alpha z$ and $[y,z] = 0$, where $\alpha \in F$.
\end{theor}
\begin{proof} Suppose that $L$ is minimal non-elementary and not solvable. Clearly $L$ is an $A$-algebra and so has the form given in Theorem \ref{t:ps}. Clearly $R \neq 0$, since, otherwise, $L$ is elementary.
\par

Choose a basis $e,f,h$ for $S_i$ such that $[h,e] = 2e$, $[h,f] = -2f$, $[e,f] =h$ and consider the subalgebra $B = Fe + Fh + R$ of $L$. This is clearly solvable and elementary, and hence $B^2$ is abelian, by \cite[Corollary 2.2]{elem2} and \cite[Lemma 3.1]{elem1}. It follows that $[e,R^2] = [e,[h,R]] = 0$. Similarly we find that $[f,R^2] = [f,[h,R]] = 0$. Use of the Jacobi identity then shows that $[h,R^2] = [h,[h,R]] = 0$, whence $[S_i,R^2] = [S_i,[h,R]] = 0$. Using the Jacobi identity again gives that $ [Fe + Ff,R] \subseteq [h,R]$, whence $[S_i,[S_i,R]] = 0$.  
\par

Thus, 
$$[S_i,R] = [[S_i,S_i],R] = [[S_i,R],S_i] + [[R,S_i],S_i] = 0.$$

This holds for all $1 \leq i \leq n$, and so $L$ is elementary, by \cite[Theorem 4.8]{frat}, a contradiction. The result follows from Theorem \ref{t:solvalgclosed}.
\end{proof}
\bigskip

The {\em abelian socle}, Asoc$L$, of $L$ is the sum of its minimal abelian ideals.

\begin{theor}\label{t:perfect} Let $L$ be a solvable Lie algebra over a perfect field $F$. Then $L$ is minimal non-elementary  if and only if 
\begin{itemize} 
\item[(i)] $L = L^2 \rtimes
Fx$, where $L^2$ is abelian and $0 \neq \phi(L) = Asoc L$ is the
biggest ideal of $L$ properly contained in $L^2$, or 
\item[(ii)] $L$ is the three-dimensional Heisenberg algebra. 
\end{itemize}
\end{theor}
\begin{proof}
As before, $L$ is an $E$-algebra and so $L^2$ is nilpotent, by \cite[Corollary 2.2]{elem2}. Suppose that $L$ is not nilpotent. Then
$\phi(L) \neq L^2$ (see, for example, \cite[section 5]{frat}), so
there is a maximal subalgebra $M$ of $L$ such that $L = L^2 + M$.
Choosing $B$ to be a subalgebra minimal with respect to the property that $L =
L^2 + B$ we have $L^2 \cap B \leq \phi(B) = 0$ (see \cite[Lemma
7.1]{frat}), so $L = L^2 \oplus B$ and $B$ is abelian. Moreover, $L^2$ is nilpotent
and elementary, and so abelian. 
\par

 Let $A$ be a minimal abelian ideal of $L$ and suppose that $A \not \subseteq \phi(L)$. Then there is a maximal subalgebra $M$ of $L$ with $A \not \subseteq M$, so $L = A \dot{+} M$. Since $M$ is elementary, $\phi(M) = 0$ and so $\phi(L) \subseteq A$, which yields $\phi(L)=0$, a contradiction. Hence Asoc$L \subseteq \phi(L)$. Let $K$ be an ideal of $L$ with Asoc$L \subset K \subseteq L^2$. If $K \neq L^2$ then $K + B$ is elementary and so splits over Asoc$L$: say, $K+B =$ Asoc$L \dot{+} C$. But now
\[ [L,C \cap L^2] = [B,C \cap L^2]  \subseteq  [\hbox{Asoc}L +C,C \cap L^2] \subseteq C \cap L^2,
\]
so $C \cap L^2$ is an ideal of $L$, whence Asoc$L \cap C \neq 0$, a contradiction. It follows that $0 \neq \phi(L) = Asoc L$ is the
biggest ideal of $L$ properly contained in $L^2$.
\par

Now suppose that dim$B > 1$ and let $x \in B$. Then $C = L^2 + Fx$ is elementary and so $L^2 \subseteq N(C)= {\rm Asoc} (C)$ by Theorem 7.4 of \cite{frat}, and $L^2$ is completely reducible as an $Fx$-module. Write $L^2 = \oplus_{i=1}^{r} C_{i}$, where $C_{i}$ is an
irreducible $Fx$-module for $1 \leq i \leq r$. Then the minimum polynomial of the restriction of ad$x$ to $C_{i}$ is irreducible for each $i$, and so $\{({\rm ad} x)|_{L^2}: x \in B \}$ is a set of
commuting semisimple operators. Let $\Omega$ be the algebraic closure of $F$ and put $L_{\Omega} = L \otimes_{F} \Omega$, and so
on. Then $\phi(L_{\Omega}) = \phi(L)_{\Omega}$, by \cite{eld}.  Also, as $F$ is perfect, $\{({\rm ad} x)|_{L^2_{\Omega}}: x \in B_{\Omega} \}$ and $\{({\rm ad} x)|_{\phi(L)_{\Omega}}: x \in B_{\Omega} \}$ are sets of simultaneously diagonalizable linear maps. Let $f_1, \ldots, f_s, c_1, \ldots, c_t$ be a basis of these common eigenvectors, where $f_i \in \phi(L)_{\Omega}$, $c_i \in L^2_{\Omega} \setminus \phi(L)_{\Omega}$. But now $$M = \bigoplus_{i=2}^s Ff_i \bigoplus_{i=1}^t Fc_i + B_{\Omega}$$ is a maximal subalgebra of $L_{\Omega}$ and $\phi(L)_{\Omega} \not \subseteq M$, a contradiction. Hence dim$B = 1$ and we have (i). 
\par

If $L$ is nilpotent then (ii) holds as in \cite[Theorem 4.5]{elem2}. The converse holds as in \cite[Theorem 4.5]{elem2}
\end{proof}
\bigskip

Note that any algebra satisfying the conditions specified in Theorem \ref{t:perfect} (i) are Lie algebras. There are also algebras satisfying these conditions which have dimension greater than three, unlike those described in Theorem \ref{t:solvalgclosed}, as the following example shows.

\begin{ex} Let $L$ be the five dimensional Lie algebra over the real field with basis $e_1, e_2, e_3, e_4, e_5$ and mutiplication given by $[e_1,e_2] = -[e_2,e_1]= e_3+e_4$, $[e_1,e_3]=-[e_3,e_1]=-e_2+e_5$,  $[e_1,e_4]=-[e_4,e_1]=e_5$, $[e_1,e_5]=-[e_5,e_1]=-e_4$, all other products being zero. Then it is straightforward to check that $L^2 = \R e_2 + \R e_3+\R e_4+\R e_5$ is abelian, and that $L$ has the unique minimal ideal $\R e_4 +\R e_5= \phi(L)$. Over the complex field this is an elementary supersolvable Lie algebra.
\end{ex}

\end{document}